\documentclass[leqno,10pt]{article}
\usepackage{a4wide,amsmath,amssymb,amsfonts,mathrsfs,exscale,graphics,amsthm,colonequals,comment,ifthen,url,breakurl,enumerate,etoolbox}

\usepackage[curve,matrix,arrow,cmtip]{xy}
\newdir^{ (}{{}*!/-3pt/\dir^{(}}   % Usage: \ar@{^{ (}->}[r]^{j}
\newdir_{ (}{{}*!/-3pt/\dir_{(}}   % Usage: \ar@{_{ (}->}[r]^{j}    
\NoComputerModernTips

\newtoggle{thesis}
\togglefalse{thesis}

\numberwithin{equation}{section}

\theoremstyle{plain}
\newtheorem{theorem}[equation]{Theorem}%[section]

\newtheorem{lemma}[equation]{Lemma}

\newtheorem{proposition}[equation]{Proposition}

\theoremstyle{definition}
\newtheorem{definition}[equation]{Definition}

\newtheorem{construction}[equation]{Construction}
\newtheorem{remark}[equation]{Remark}
\newtheorem{example}[equation]{Example}

\DeclareMathOperator{\G}{G}

\DeclareMathOperator{\Img}{im}

\let\into\hookrightarrow

\newcommand{\defeq}{\colonequals}

\newcommand{\Gm}[1][\empty]{
  \ifthenelse{\equal{#1}{\empty}}
    {\mathbb{G}_m}
    {\mathbb{G}_{m,#1}}}

 \newcommand{\Gred}[1][\empty]{
  \ifthenelse{\equal{#1}{\empty}}
    {G^{\text{red}}}
    {G^{\text{red},#1}}}

 \newcommand{\Rep}[1][\empty]{
  \ifthenelse{\equal{#1}{\empty}}
    {\mathop{\text{\tt Rep}}\nolimits}
    {\mathop{\text{$#1$-{\tt Rep}}}\nolimits}}

\DeclareMathOperator{\gr}{gr}
\DeclareMathOperator{\rk}{rk}

\newcommand\toover[1]{\mathrel{\smash{\overset{#1}{\to}}}}
\newcommand\varto[1]{\mathrel{\hbox to #1pt{\rightarrowfill}}}
\renewcommand{\implies}{\Rightarrow}

\newcommand{\BQ}{{\mathbb{Q}}}

\newcommand{\BZ}{{\mathbb{Z}}}

\newcommand{\CI}{{\mathcal I}}

\newcommand{\CO}{{\mathcal O}} 
\newcommand{\CP}{{\mathcal P}}

\newcommand{\CZ}{{\mathcal Z}}

\newcommand{\Gscr}{{\mathscr G}}

\newcommand{\leftexp}[2]{{\vphantom{#2}}^{#1}{#2}}

\let\phi\varphi

\DeclareMathOperator{\GL}{GL}

\begin{document}
\title{$p$-kernels occurring in an isogeny class of $p$-divisible groups}

\author{Paul Ziegler\footnote{Imperial College London,
 {\tt p.ziegler@imperial.ac.uk}}
}

%\date{Dezember 8, 2011}

\maketitle
\abstract{We give a criterion which allows to determine, in terms of the combinatorics of the root system of the general linear group, which $p$-kernels occur in an isogeny class of $p$-divisible groups over an algebraically closed field of positive characteristic. As an application we obtain a criterion for the non-emptiness of certain affine Deligne-Lusztig varieties associated to the general linear group.}

\section{Introduction}

This article studies the relationship between two invariants of a $p$-divisible group $\Gscr$ over an algebraically closed field of characteristic $p>0$: The first is the isogeny class of $\Gscr$ which is encoded in its Newton polygon and the second is the isomorphism class of the kernel of multiplication by $p$ on $\Gscr$. Once certain numerical invariants of $\Gscr$ are fixed, both these invariants can only take on finitely many values. In this article, we give a computable criterion, in terms of the combinatorics of the root system of the general linear group, which determines which pairs of these invariants can occur together for some $\Gscr$. That is we determine which $p$-kernels can occur in any isogeny class of $p$-divisible groups. We also consider the analogous question in equal characteristic.

This question is motivated by our interest in the stratifications of suitable moduli spaces of abelian varieties or $p$-divisible groups obtained by decomposing these spaces according to the two invariants described above. For example, on a Rapoport-Zink space (c.f. \cite{RapoportZink}), one can define the Ekedahl-Oort stratification by decomposing the space according to the isomorphism class of the $p$-kernel of the universal $p$-divisible group and our criterion allows to determine which of these strata are non-empty. Similarly, on a moduli spaces of abelian varieties with suitable extra structure in positive characteristic, one obtains two stratifications, the Newton polygon stratifications and the Ekedahl-Oort stratification and we would like to understand which strata of these two stratifications intersect each other. However, in this context one encounters not just $p$-divisible groups, but $p$-divisible groups with additional structure such as a pairing. For applications to such stratifications it would thus be necessary to obtain generalizations of the results of this article for $p$-divisible groups with such additional structure. It seems natural to expect that in such a setting the analogues of our results should hold with the group $\GL_h$ replaced by an arbitrary reductive group. The author intends to treat this question in a follow-up article.

As an another application of our results, in Section \ref{DLSection} we give a criterion for the non-emptiness of affine Deligne-Lusztig varieties for the group $\GL_h$ in the situation where the involved Hodge cocharacter is minuscule.
\medskip

Throughout, we work with Dieudonn\'e modules instead of $p$-divisible groups. We work over a fixed algebraically closed field $k$ of characteristic $p$ and work either over the Witt ring $\CO=W(k)$ or $\CO=k[[t]]$ whose uniformizer $p$ or $t$ we denote by $\epsilon$. We use the following language: A Dieudonn\'e module is a finite free module over $\CO$ together with suitably semilinear endomorphisms $F$ and $V$ satisfying $FV=FV=\epsilon$. A $1$-truncated Dieudonn\'e module is a finite-dimensional vector space over $k$ together with suitably semilinear endomorpism $F$ and $V$ satisfying $\ker F=\Img V$ and $\Img F=\ker V$. To each Dieudonn\'e module $M$ one can associate its truncation $M/\epsilon M$. By a lift of a $1$-truncated Dieudonn\'e module $Z$ we mean a Dieudonn\'e module $M$ together with an isomorphism $M/\epsilon M\cong Z$. To each Dieudonn\'e module $M$ we associate the Newton polygon obtained via covariant Dieudonn\'e theory. Then we answer the above question by determining for a given $1$-truncated Dieudonn\'e module $Z$ and Newton polygon $\CP$ whether there exists a lift of $Z$ with Newton polygon $\CP$. 

\par
For the sake of simplicity, in this introduction we restrict ourselves to the case that $\CP$ is the straight Newton polygon with slope $n/(n+m)$ and endpoint $(n+m,n)$ for some non-negative coprime integers $n$ and $m$. For the result for arbitrary Newton polygons see Theorem \ref{MainTheorem}. To state our result, we will need the following: 

Let $h\defeq n+m$ and $G\defeq \GL_{h,\CO}$. Let $T\subset G$ be the torus of diagonal matrices and $B\subset G$ the Borel subgroup of upper triangular matrices. Let $W\cong S_n$ be the Weyl group of $G$ with respect to $T$ and $S=\{(i,i+1)\mid 1\leq i\leq h-1\}$ the generating system of $W$ induced by $B$. Let $\mu\in X_*(T)$ be the cocharacter $t\mapsto (t,\hdots,t,1,\hdots,1)$ where $t$ occurs with multiplicity $m$. Let $I$ be the type $S\setminus \{(m,m+1)\}$. We denote by $W_I\subset W$ the subgroup generated by $I$ and by $\leftexp{I}{W}\subset W$ the set of left reduced elements with respect to $W_I$. There exists a natural bijection between isomorphism classes of $1$-truncated Dieudonn\'e modules $Z$ satisfying $\rk_k Z=h$ and $\rk_kF(Z)=n$ and elements of $\leftexp{I}{W}\subset W$ (c.f. Subsection \ref{TruncClassSect}). For $w\in \leftexp{I}{W}$ we denote the corresponding $1$-truncated Dieudonn\'e module by $Z_w$.

Let $\CI\subset G(\CO)$ be the preimage of $B(k)$ under the projection $G(\CO)\to G(k)$. Let $\tilde W$ be the extended Weyl group of $G$. We denote the canonical inclusion $X_*(T)\into \tilde W$ by $\lambda\mapsto \epsilon^\lambda$. For $\lambda\colon t\mapsto (t^{\lambda_1},\hdots,t^{\lambda_h}) \in X_*(T)$ we let $\eta_{\lambda}$ be the unique permutation $\eta\in W$ such that $\lambda_{\eta(1)}\leq \hdots \lambda_{\eta(h)}$ and $\eta(i)\leq \eta(i')$ for any $i\leq i'$ such that $\lambda_i=\lambda_{i'}$. Finally, we let $x_{n,m}\in \tilde W$ be the matrix of Frobenius on the minimal Dieudonn\'e module $H_{n,m}$ (c.f. Definition \ref{xnmDef}). Then our result is:
\begin{theorem}[{c.f. Theorem \ref{MainTheorem}}] \label{MainTheoremIntro}
  Let $w\in \leftexp{I}{W}$. The following are equivalent:
  \begin{enumerate}[(i)]
  \item The $1$-truncated Dieudonn\'e module $Z_w$ admits a lift with Newton polygon $\CP$.
  \item There exist $\lambda\in X_*(T)$ satisfying $\epsilon^{-\lambda}x_{n,m}\epsilon^\lambda\in W\epsilon^\mu W$ as well as $y\in W$ such that $ww_0w_{0,I}\epsilon^\mu\in \CI y\CI \eta_\lambda^{-1}\epsilon^{-\lambda}x_{n,m}\epsilon^\lambda\eta_\lambda \CI y^{-1} \CI$.
  \end{enumerate}
\end{theorem}
Let $\CZ$ denote the center of $G$. The group $X_*(\CZ)\subset X_*(T)$ acts on the set $$\{\lambda\in X_*(T)\mid \epsilon^{-\lambda}x_{n,m}\epsilon^\lambda\in W\epsilon^\mu W\}$$ by addition. By Lemma \ref{LambdaFiniteness}, this action has finitely many orbits. In this way the existence quantifier in $(ii)$ ranges over a finite set. Hence condition $(ii)$ is computable.

\medskip

Now we explain our argument: 

Given a isosimple Dieudonn\'e module $M$ of slope $n/(n+m)$, we obtain a filtration $(G^jZ)_{j\in \BZ}$ on $Z\defeq M/\epsilon M$ such that for all $j\in \BZ$ we have $F(G^jZ)=G^{j+n}Z\cap F(Z)$ and $V(G^jZ)=G^{j+m}Z\cap V(Z)$ by embedding $M$ into the minimal Dieudonn\'e module $M_{n,m}$  (c.f. Subsection \ref{CompFilConsSec}). Conversely, given such a filtration on a $1$-truncated Dieudonn\'e module $Z$ we can construct a lift of $Z$ which is isoclinic of slope $n/(n+m)$ (c.f. Subsection \ref{LiftConsSec}). Hence, in order to determine whether a given $Z$ admits such a lift, it suffices to determine whether there exists such a filtration on $Z$, which we call a compatible filtration of type $(n,m)$ (c.f. Subsection \ref{CompFilSec}).

To determine whether there exists a compatible filtration on $Z$ of type $(n,m)$, we first consider the associated graded situation: Given a compatible filtration $(G^j Z)_{j\in \BZ}$ of type $(n,m)$, one obtains the graded $1$-truncated Dieudonn\'e module $\oplus_{j\in \BZ}G^jZ/G^{j+1}Z$ on which $F$ and $V$ act as morphisms of degree $n$ and $m$ respectively. Following an idea of Chen and Viehmann, in Subsection \ref{LambdaClassSection}, we classify such graded $1$-truncated Dieudonn\'e modules in terms of cocharacters $\lambda\in X_*(T)$ satisfying $\epsilon^{-\lambda}x_{n,m}\epsilon^\lambda\in W\epsilon^\mu W$.

 Then, by comparing compatible filtrations to the associated gradings, we obtain the following criterion for the existence of a compatible filtration:

\begin{theorem}[{c.f. Theorem \ref{CompFlagExists}}] \label{CompFlagExistsIntro}
  Let $M$ be a Dieudonn\'e module of rank $h$ such that $M/FM$ has length $m$. The following are equivalent:
  \begin{enumerate}[(i)]
  \item On the truncation $Z=M/\epsilon M$ there exists a compatible filtration of type $(n,m)$.
  \item There exists $\lambda\in X_*(T)$ satisfying $\epsilon^{-\lambda}x_{n,m}\epsilon^\lambda\in W\epsilon^\mu W$ such that the matrix of $F\colon M\to M$ with respect to some $\CO$-basis of $M$ lies in $\CI\eta_\lambda^{-1}\epsilon^{-\lambda}x_{n,m}\epsilon^\lambda\eta_\lambda \CI$.
  \end{enumerate}
\end{theorem}

Then by combining the above steps we obtain Theorem \ref{MainTheoremIntro}.

\medskip

\paragraph{Acknowledgement} I am very grateful to Richard Pink for numerous conversations on the topic of this article. I also thank Torsten Wedhorn for helpful remarks and conversations. This work was supported by a fellowship of the Max Planck society as well as a fellowship of the Swiss National Fund. Part of this work was carried out during a visit to the FIM at ETH Z\"urich. I thank the institute for its hospitality and excellent working conditions.

\section{Preliminaries}

\subsection{Setup}

\label{sec:setup}
Throughout, we will work with the following setup and notation:

\begin{itemize}
\item $k$ is an algebraically closed field of characteristic $p>0$.
\item $\CO$ is either the Witt ring $W(k)$ or the ring $k[[t]]$.
\item For $a\in k$, we let $[a]\in \CO$ be either the canonical lift of $a$ in $W(k)$ or the image of $a$ under the inclusion $k\into k[[t]]$.
\item $\epsilon \in \CO$ is the uniformizer $p$ or $t$ accordingly.
\item $L$ is the function field of $\CO$.
\item $v\colon L\to \BZ$ is the valuation normalized such that $v(\epsilon)=1$.
\item $\sigma\colon \CO \to \CO$ is either the canonical lift $W(k)\to W(k)$ of Frobenius or the automorphism $k[[t]]\to k[[t]]$ fixing $t$ and sending $a\in k$ to $a^p$.
\item A Dieudonn\'e module is a finite free $\CO$-module together with a $\sigma$-linear endomorphism $F$ and a $\sigma^{-1}$-linear endomorphism $V$ satisfying $FV=VF=\epsilon$. (In the equicharacteristic case, such an object is usually called an effective and minuscule local $\GL_h$-shtuka.)
\item $F_k\colon k\to k,\; x\mapsto x^p$ is the Frobenius automorphism.
\item A $1$-truncated Dieudonn\'e module is a finite-dimensional $k$-vector space together with an $F_k$-linear endomorphism $F$ and an $F_k^{-1}$-linear endomorphism $V$ such that $\Img F=\ker V$ and $\ker V=\Img V$.
\item To a Dieudonn\'e module $M$ we associate the $1$-truncated Dieudonn\'e module $M/\epsilon M$.
\item  By the Newton polygon of a Dieudonn\'e module $M$ we mean the Newton polygon obtained via covariant Dieudonn\'e theory. That is a Dieudonn\'e module is isoclinic of slope $r/s$ for integers $r,s\geq 0$ if and only if it is isogenous to a Dieudonn\'e module on which $\epsilon^{-r}F^s$ is an automorphism. 
\item We write Newton polygons in the form $\CP=(\nu_1,\hdots,\nu_N)$ where $\nu_1\leq \hdots \leq \nu_N$ are the slopes occurring in $\CP$ with multiplicities.
\item For $\nu\in \BQ^{\geq 0}$ we denote by $n_\nu$ and $m_\nu$ the unique non-negative coprime integers such that $\nu=n_\nu/(n_\nu+m_\nu)$.

\end{itemize}

We will often work with respect to given integers $0\leq d\leq h$. Then we use the following:
\begin{itemize}
\item $G$ is the group scheme $\GL_{h,\CO}$.
\item $T\subset G$ is the canonical torus of diagonal matrices.
\item $B\subset G$ is the canonical Borel subgroup of upper triagonal matrices.
\item $\CI\subset G(\CO)$ is the preimage of $B(k)$ under the projection $G(\CO)\to G(k)$.
\item $G(\CO)_1$ is the kernel of the projection $G(\CO)\to G(k)$.
\item $W\cong S_h$ is the Weyl group of $G$ with respect to $T$ which we identify with the set of monomial matrices with entries in $\{0,1\}$ in either $G(k)$ or $G(\CO)$.
\item $S=\{(i,i+1)\mid 1\leq i\leq h-1\}\subset W$ is the set of simple reflections induced by $B$.
\item $I\subset S$ is the type $S\setminus \{(h-d-1,h-d)\}$.
\item $W_I\subset W$ is the subgroup generated by $I$.
\item $\leftexp{I}{W}$ is the set of left reduced elements with respect to $I$, that is the set of elements $w$ which have minimal length in $W_Iw$.
\item $w_0$ is the longest element in $W$.
\item $w_{0,I}$ is the longest element in $W_I$.
\item We denote by $\tilde W\cong X_*(T) \rtimes W$ the extended Weyl group of $G$, which we identify with the group of monomial matrices in $G(\CO)$ with entries in $\{0\}\cup p^\BZ$.
\item For $\lambda\in X_*(T)$, we denote by $\epsilon^\lambda\defeq \lambda(p)$ its image in $\tilde W$.
\item We denote the cocharacter $\lambda\in X_*(T)$ which sends $t\in \Gm$ to the diagonal matrix with entries $(t^{\lambda_1},\hdots,t^{\lambda_h})$ by $(\lambda_1,\hdots,\lambda_h)$.
\item $\mu\in X_*(T)$ is the cocharacter $(1,\hdots,1,0,\hdots,0)$ where the entry $1$ has multiplicity $h-d$.
\item We say that a Dieudonn\'e module $M$ has Hodge polygon given by $\mu$ if $\rk_\CO M=h$ and $M/FM$ has length $d$.
\item We denote again by $\sigma$ the automorphism of $G(\CO)$ induced by $\sigma\colon \CO\to \CO$.
\item To an element $g\in G(\CO)\epsilon^\mu G(\CO)$ we associate the Dieudonn\'e module $M_g\defeq (\CO^h,g\sigma)$. This gives a bijection between $G(\CO)$-$\sigma$-conjugacy classes in $G(\CO)\epsilon^\mu G(\CO)$ (i.e. orbits under the action $G(\CO)\times G(\CO)\to G(\CO), (g,h)\mapsto gh\sigma(h)^{-1}$) and isomorphism classes of Dieudonn\'e modules with Hodge polygon given by $\mu$.
\end{itemize}

\subsection{Classification of $1$-truncated Dieudonn\'e modules} \label{TruncClassSect}
Fix integers $0\leq d\leq h$. We call a $1$-truncated Dieudonn\'e module $Z$ \emph{of numerical type $(d,h)$} ifit satisfies $\rk_k Z=h$ and $\rk_k  F(Z)=h-d$. Any $1$-truncated Dieudonn\'e module of numerical type $(d,h)$ can be lifted to a Dieudonn\'e module with Hodge polygon given by $\mu$. Furthermore, one can check that for two elements $g_1,g_2\in G(\CO)\epsilon^\mu G(\CO)$ the truncations $M_{g_1}/\epsilon M_{g_1}$ and $M_{g_2}/\epsilon M_{g_2}$ are isomorphic as $1$-truncated Dieudonn\'e modules if and only if $g_2$ is $G(\CO)$-$\sigma$-conjugate to an element of $\G(\CO)_1\epsilon^\mu G(\CO)_1$. Hence isomorphism classes of $1$-truncated Dieudonn\'e modules of numerical type $(d,h)$ correspond to the $\G(\CO)$-$\sigma$-conjugacy classes in $G(\CO)_1 \backslash G(\CO)\epsilon^\mu G(\CO) /G(\CO)_1$. By \cite[Theorem 1.1]{ViehmannTruncations} the set $\{ww_0w_{0,I}\epsilon^\mu\mid w\in \leftexp{I}{W} \}$ gives a set of representatives for these conjugacy classes. Thus the $1$-truncated Dieudonn\'e modules $Z_w\defeq M_{ww_0w_{0,I}\epsilon^\mu}/\epsilon M_{ww_0w_{0,I}\epsilon^\mu}$ for $w\in \leftexp{I}{W}$ are representatives for the isomorphism classes of $1$-truncated Dieudonn\'e modules.

\subsection{Minimal Dieudonn\'e modules}
  For coprime non-negative integers $n$ and $m$, the minimal Dieudonn\'e module $H_{n,m}$ of slope $n/(n+m)$ is defined as follows (c.f. \cite{OortMinimal}): It is the free $\CO$-module with basis $e_1,\hdots,e_{n+m}$. For $i> n+m$, we write $i=a(n+m)+b$ for unique integers $a>1$ and $1\leq b \leq n+m$ and define $e_i\defeq \epsilon^a e_b$. Then $F$ and $V$ are defined by $F(e_i)=e_{i+n}$ and $V(e_i)=e_{i+m}$ for all $i\geq 1$. 

Let $\Phi$ be the $\sigma$-semilinear automorphism of $H_{n,m}$ which fixes the $e_i$. Then $\Phi \pi=\pi \Phi$, $F=\Phi \pi^n$ and $V=\Phi^{-1} \pi^m$.

\begin{definition} \label{xnmDef}
  Let $n$ and $m$ be coprime non-negative integers. We define $x_{n,m}\in \tilde W$ to be the matrix of $F\colon H_{n,m}\to H_{n,m}$ with respect to the basis $(e_h,\hdots,e_1)$. 
\end{definition}

\section{Graded $1$-truncated Dieudonn\'e modules} \label{GTDMSec}
Throughout this section we fix coprime non-negative integers $n$ and $m$ and let $h\defeq n+m$ and $d \defeq n$.

By a grading of a vector space we will always mean a $\BZ$-grading. For a graded vector space $X=\oplus_{j\in \BZ}X^j$ we will call the elements of the $X^j$ the homogenous elements of $X$. For $i\in \BZ$, we say that an additive homomorphism $X\to X'$ between graded vector spaces is \emph{of degree $i$} if it sends every homogenous element of degree $j$ to a homogenous element of degree $j+i$.

\begin{definition}
  A \emph{graded $1$-truncated Dieudonn\'e module} is a $1$-truncated Dieudonn\'e module $Z$ together with a grading $Z=\oplus_{j\in\BZ} Z^j$ such that $F$ and $V$ send homogenous elements of $Z$ to homogenous elements.

A morphism of graded $1$-truncated Dieudonn\'e modules is a morphism of $1$-truncated Dieudonn\'e modules of degree zero.
\end{definition}

\begin{definition}
  A \emph{graded $1$-truncated Dieudonn\'e module of type $(n,m)$} over $k$ is a graded $1$-truncated Dieudonn\'e module $(Z,F,V)$ such that $F$ is of degree $n$ and $V$ is of degree $m$.
\end{definition}

\begin{lemma} \label{MuLemma}
  Let $Z=\oplus_{j\in \BZ}Z^j$ be a graded $1$-truncated Dieudonn\'e module of type $(n,m)$. There exists an integer $c$ such that $\rk_k Z=c(n+m)$ and such that for every $j\in \BZ$ we have
  \begin{equation*}
    \sum_{i\equiv j \pod{n+m}}\rk_k Z^i=c.
  \end{equation*}
\end{lemma}
\begin{proof}

 For $j\in \BZ$ let $Z(j)\defeq \oplus_{i\equiv j \pod{n+m}}Z^i$. The fact that $Z$ is graded of type $(n,m)$ implies that for each $j$ we have a short exact sequence
 \begin{equation*}
   0\to Z(j-m)/(Z(j-m)\cap F(Z)) \toover{V} Z(j) \toover{F} Z(j+n)\cap F(Z) \to 0.
 \end{equation*}
Using $Z(j-m)=Z(j+n)$ this implies $\rk_k Z(j)=\rk_k Z(j+n)$. Since $n$ and $n+m$ are coprime, iterating this fact yields the claim.

\end{proof}

\subsection{Classification in terms of semimodules}

\begin{definition}[{c.f. \cite[(1.7)]{OortSimple} and \cite[Section 6]{JOPurity}}]
  A \emph{beginning of a semi-module of type $(n,m)$} is a subset $C\subset \BZ$ such that for each $i\in \BZ$ the equivalence class $i+(n+m)\BZ$ contains exactly one element of $C$ and for each $i\in C$ either $i+n\in C$ or $i-m\in C$.
\end{definition}

\begin{lemma}
  Let $Z=\oplus_{j\in \BZ}Z^j$ a graded $1$-truncated Dieudonn\'e module of type $(n,m)$ of rank $h$. Then $C_Z\defeq \{j\in \BZ\mid Z^j\not= 0\}$ is a beginning of a semi-module of type $(n,m)$.
\end{lemma}
\begin{proof}
  This follows from the definition of $1$-truncated Dieudonn\'e modules of type $(n,m)$ together with Lemma \ref{MuLemma}.  
\end{proof}

\begin{construction}
  Let $C$ be a beginning of a semi-module of type $(n,m)$. We construct a graded $1$-truncated Dieudonn\'e module $Z_C$ of type $(n,m)$ and of rank $h$ as follows:

Let $Z$ be the free $k$-vector space with basis $(e_j)_{j\in C}$. Endow $Z_C$ with the grading for which each $e_j$ is homogenous of degree $j$. We define $F$ and $V$ as follows: Let $j\in C$. If $j+n\in C$ we let $F(e_j)\defeq e_{j+n}$ and $V(e_{j+n})\defeq 0$. Otherwise $j-m\in C$ and we let $V(e_{j-m})\defeq e_j$ and $F(e_j)=0$. Then by a direct verification $Z_C$ has the required properties.
\end{construction}
\begin{proposition} \label{ShortClassification1}
  The assignments $Z=\oplus_{j\in \BZ}Z^j\mapsto C_Z$ and $C\mapsto Z_C$ give mutually inverse bijections between the set of isomorphism classes of $1$-truncated Dieudonn\'e modules of type $(n,m)$ and of rank $h$ and the set of beginnings of semi-modules of type $(n,m)$. 
\end{proposition}
\begin{proof}
The identity $C=C_{Z_C}$ follows directly from the definition of $Z_C$.

 It remains to prove that each $Z$ is isomorphic to $Z_{C_Z}$ as a graded $1$-truncated Dieudonn\'e module. To see this, start with an element $j_0\in C_Z$ and a non-zero element $f_0 \in Z^{j_0}$. We iteratively construct a sequence of pairs $(j_s\in C_Z, f_s\in Z^{j_s}\setminus \{0\})$ as follows: If $j_k+n\in C$ we let $j_{k+1}\defeq j_k+n$ and $f_{j+1}\defeq F(f_j)$. Otherwise we let $j_{k+1}\defeq j_k-m$ and $f_{j+1}\in Z^{j_k-m}$ the unique element such that $V(f_{j+1})=f_j$. 

 By construction, for $s\geq 0$, the element $j_s\in C_Z$ is the unique element of $C_Z$ in $j_s+sn+h\BZ$. Thus $j_h=j_0$ and hence $f_h=\lambda f_0$ for some $\lambda\in k^*$. Pick $\mu\in k^*$ such that $\mu^{p^{m-n}}\lambda=\mu$. In $C_Z$ there a $m$ elements $j$ satisfying $j+n\in C_Z$ and $n$ elements $j$ satisfying $j-m\in C_Z$ (c.f. \cite[Section 6]{JOPurity}). Hence by replacing $f_0$ by $\mu f_0$ in the above construction we obtain a sequence such that $f_h=f_0$. Then for each $j\in C_Z$ we let $e_j\defeq f_k$ for the unique $0\leq k< h$ such that $j_k=j$. The resulting basis $(e_j)_{j\in C_Z}$ of $Z$ gives an isomorphism $Z\cong Z_{C_Z}$ of graded $1$-truncated Dieudonn\'e modules.
\end{proof}

\subsection{Classification in terms of cocharacters} \label{LambdaClassSection}
Now we show that $1$-truncated Dieudonn\'e modules of type $(n,m)$ and rank $h$ can also be classified by certain cocharacters $\lambda\in X_*(T)$. The idea behind this classification is due to Chen and Viehmann (c.f. \cite{ChenViehmann}).
\begin{construction} \label{ZlambdaCons}
  Let $\lambda=(\lambda_1,\hdots,\lambda_h) \in X_*(T)$ be a cocharacter satisfying $\epsilon^{-\lambda} x_{n,m} \epsilon^{\lambda}\in W\epsilon^\mu W$. We construct a graded $1$-truncated Dieudonn\'e module of type $(n,m)$ as follows: As in Definition \ref{xnmDef}, we consider the Dieudonn\'e module $M_{n,m}$ with the basis $(e_{n+m},\hdots,e_1)$. For $1\leq j\leq h$ let $f_j\defeq \epsilon^{\lambda_j}e_{h+1-j}$. The $f_j$ form a $\CO$-basis of a submodule $M\subset M_{n,m}$ and the matrix of $F$ with respect to this basis is $\epsilon^{-\lambda} x_{n,m} \epsilon^{\lambda}$. Hence the assumption $\epsilon^{-\lambda} x_{n,m} \epsilon^{\lambda}\in W\epsilon^\mu W$ means that $M$ is a sub-Dieudonn\'e module of $M_{n,m}$ with Hodge polygon given by $\mu$. Let $Z\defeq M/\epsilon M$ with basis $(\bar f_j\defeq f_j+\epsilon M)_{1\leq j\leq h}$. Equipping $Z$ with the grading for which each $\bar f_j$ is homogenous of degree $h+1-j+h\lambda_j$ makes $Z$ into a graded $1$-truncated Dieudonn\'e module of type $(n,m)$ which we denote by $Z_\lambda$.
\end{construction}

\begin{proposition} \label{ShortClassification2}
  The assignment $\lambda\mapsto Z_\lambda$ gives a bijection from the set
  \begin{equation*}
    \{\lambda\in X_*(T)\mid \epsilon^{-\lambda} x_{n,m} \epsilon^{\lambda}\in W\epsilon^\mu W\}   
      \end{equation*}
to the set of isomorphism classes of $1$-truncated Dieudonn\'e modules of type $(n,m)$ and rank $h$.
\end{proposition}
\begin{proof}
  Let $Z$ be a $1$-truncated Dieudonn\'e module of type $(n,m)$ and rank $h$. By Proposition \ref{ShortClassification1} we may assume that $Z=Z_C$ for some beginning of a semi-module $C$. For each $1\leq j\leq h$ let $\lambda_j$ be the unique integer such that $h+1-j+h\lambda_j\in C$. Let $M$ be the the sub-$\CO$-module of $H_{n,m}$ spanned by $\{e_j\mid j\in C\}=\{\epsilon^{\lambda_j}e_{h+1-j}\mid 1\leq j\leq h\}$. The fact that $C$ is the beginning of a semi-module of type $(n,m)$ implies that $M$ is a sub-Dieudonn\'e module of $H_{n,m}$. Furthermore, the assignment $e_j\in M\mapsto e_j\in Z_C$ for $i\in C$ induces an isomorphism $M/\epsilon M\cong Z_C$ of $1$-truncated Dieudonn\'e modules. This implies that $M$ has Hodge polygon given by $\mu$ which in turn is equivalent to $\epsilon^{-\lambda} x_{n,m} \epsilon^{\lambda}\in W\epsilon^\mu W$. It follows from the above that $Z_C\cong Z_\lambda$ as graded $1$-truncated Dieudonn\'e modules. Thus the map in question is surjective. As for the injectivity, it follows directly from Construction \ref{ZlambdaCons} that $\lambda$ can be recovered from the grading on $Z_\lambda$.
\end{proof}

\section{Compatible filtrations on $1$-truncated Dieudonn\'e modules} \label{CompFilSection}

\subsection{Definitions} \label{CompFilSec}
By a decreasing filtration $(G^j X)_{j\in \BZ}$ on a finite-dimensional vector space $X$ we mean a family of subspaces such that $G^jX\supset G^{j+1}X$ for all $j\in \BZ$, such that $G^jX=X$ for all small enough $j$ and such that $G^jX=0$ for all large enough $j$. Given two descending filtrations $(G^jX)_{j\in \BZ}$ and $(\G^jX')_{j\in \BZ}$ on two such vector spaces $X$ and $X'$ and an integer $i$, we call an additive homomorphism $h\colon X\to X'$ filtered of degree $i$ if $h(G^jX)\subset G^{j+i}X$ for all $j\in \BZ$.

\begin{lemma} \label{CompFilLemma}
  Let $n$ and $m$ be coprime non-negative integers and $Z$ a $1$-truncated Dieudonn\'e module over $k$. Let $(G^jZ)_{j\in \BZ}$ a descending filtration on $Z$ such that $F$ is filtered of degree $n$ and such that $V$ is filtered of degree $m$. The following two conditions are equivalent:
  \begin{enumerate}[(i)]
  \item The vector space $\gr Z\defeq \oplus_j G^j Z/G^{j+1}Z$ together with the graded semilinear endomorphisms of degree $n$ and $m$ induced by $F$ and $V$ is a graded $1$-truncated Dieudonn\'e module of type $(n,m)$.
  \item For all $j\in \BZ$ we have $F(G^jZ)=G^{j+n}Z\cap F(Z)$ and $V(G^jZ)=G^{j+m}Z\cap V(Z)$.
  \end{enumerate}
\end{lemma}
\begin{proof}
  This follows from a direct verification.
\end{proof}
\begin{definition}
  Let $n$ and $m$ be coprime non-negative integers and $Z$ a $1$-truncated Dieudonn\'e module over $k$. A \emph{compatible filtration of type $(n,m)$} on $Z$ is a decreasing filtration $E=(\G^jZ)_{j\in \BZ}$ by $k$-submodules such that $F$ is filtered of degree $n$, such that $V$ is filtered of degree $m$ and such that the equivalent conditions of Lemma \ref{CompFilLemma} are satisfied.

For such an $E$, we denote by $\gr_E(Z)$ the associated graded $1$-truncated Dieudonn\'e module from Lemma \ref{CompFilLemma}.
%A morphism of filtered $1$-truncated Dieudonn\'e modules is a morphism of $1$-truncated Dieudonn\'e modules which is filtered of degree zero.
\end{definition}
\begin{example} \label{AssFiltered}
  Let $n$ and $m$ be coprime non-negative integers and $Z=\oplus_{i\in \BZ}Z^i$ a graded $1$-truncated Dieudonn\'e module of type $(n,m)$. Then the filtration $E$ given by $G^j(Z)\defeq \oplus_{i\geq j}Z^i$ is a compatible filtration of type $(n,m)$. The associated graded $1$-truncated Dieudonn\'e module $\gr_E Z$ is canonically isomorphic to $Z$.
\end{example}

\begin{definition}
  Let $\CP=(\nu_1,\hdots,\nu_N)$ a Newton polygon. Let $Z$ be a $1$-truncated Dieudonn\'e module. A \emph{compatible filtration with Newton polyon $\CP$} on $Z$ is a filtration $0=Z_0\subset Z_1 \hdots \subset Z_N=Z$ by sub-$1$-truncated Dieudonn\'e modules such that the subquotients $Z_i/Z_{i-1}$ are $1$-truncated Dieudonn\'e modules of rank $n_{\nu_i}+m_{\nu_i}$ together with compatible filtrations $E_i$ on the $Z_i/Z_{i-1}$ of type $(n_{\nu_i},m_{\nu_i})$.
\end{definition}

\subsection{Compatible filtrations associated to Dieudonn\'e modules}
\label{CompFilConsSec}
In this subsection, for a Dieudonn\'e module $M$ with Newton polygon $\CP$ we construct a compatible filtration with Newton polygon $\CP$ on $M/\epsilon M$. The idea behind this construction is originally due to Manin (c.f. \cite[Section III.5]{Manin}) and was also used by de Jong and Oort in \cite{JOPurity} and by Oort in \cite{OortSimple}.

\begin{construction} \label{FilRedCons1}
  Let $n$ and $m$ be coprime non-negative integers. Let $M$ be an isosimple Dieudonn\'e module of slope $n/(n+m)$. We define a compatible filtration of type $(n,m)$ on the $1$-truncated Dieudonn\'e module $Z\defeq M/pM$ as follows: 

By the slope assumption there exists an embedding $M\into H_{n,m}$. We choose such an embedding and let $M^j\defeq M\cap \pi^j H_{n,m}$ for all $j\geq 0$. The fact that $F=\pi^n \Phi$ and $V=\pi^m \Phi^{-1}$ on $H_{n,m}$ implies that $F(M^j)=M^{j+n}\cap F(M)$ and $V(M^j)=M^{j+m}\cap V(M)$ for all $j\in \BZ$. These two identities imply that $G^j(Z)\defeq (M^j+\epsilon M)/\epsilon M\subset Z$ defines a compatible filtration $E_M$ of type $(n,m)$ on $Z$. Since $M$ is isosimple, the vector spaces $Z$ and $\gr_E Z$ have rank $n+m$.

%Associating to such an $M$ the filtered $1$-truncated Dieudonn\'e module $(M/pM,E_M)$ is functorial in $M$.
\end{construction}
\begin{remark}
  By \cite[Section 5.6]{JOPurity} a different choice of embedding $M\into H_{n,m}$ in Construction \ref{FilRedCons1} yields to a filtration which differs from the given one only by a shift of the indexing of the filtration.
\end{remark}
\begin{construction} \label{FilRedCons2}
Let $M$ be a Dieudonn\'e module and $Z\defeq M/\epsilon M$. Let $\CP$ be the Newton polygon of $M$. We define a compatible filtration on $Z$ as follows: We start with the slope filtration of $M$ (c.f. e.g. \cite[Corollary 13]{ZinkSF}) and refine it to a filtration $0=M_0\subset M_1\subset \hdots \subset M_N=M$ by sub-Dieudonn\'e modules such that each $M_i/M_{i-1}$ is isosimple. For $0\leq i\leq N$ let $Z_i\defeq M_i/\epsilon M_i$. Then Construction \ref{FilRedCons1} applied to the Dieudonn\'e modules $M_i/M_{i-1}$ yields compatible filtrations $E_i$ on $Z_i/Z_{i-1}\cong (M_i/M_{i-1})/\epsilon (M_i/M_{i-1})$. Alltogether we obtain a compatible filtration with Newton polygon $\CP$.
\end{construction}

\subsection{Lifts associated to compatible filtrations} \label{LiftConsSec}
In Construction \ref{FilRedCons2}, we associate to each Dieudonn\'e module $M$ a compatible filtration $E_M$ on $M/\epsilon M$ with the same Newton polygon as $M$. In this subsection we show that conversely, for each $1$-truncated Dieudonn\'e module $Z$ together with a compatible filtration $E$ on $Z$ with Newton polygon $\CP$ there exists a Dieudonn\'e module $M$ lifting $Z$ which has Newton polygon $\CP$.

\begin{construction} \label{LiftCons1}
 Let $\CP=(\nu_1,\hdots,\nu_N)$ be a Newton polygon. Let $Z$ be a $1$-truncated Dieudonn\'e module and $E=((Z_i)_{0\leq i\leq N},(E_i)_{1\leq i\leq N})$ a compatible filtration with Newton polygon $\CP$ on $Z$. We construct a Dieudonn\'e module $M$ lifting $Z$ as follows:

For each $1\leq i\leq N$ let $C_i$ be the beginning of a semi-module of type $(n_i,m_i)\defeq (n_{\nu_i},m_{\nu_i})$ associated to $\gr_{E_i}(Z_i/Z_{i-1})$. By Proposition \ref{ShortClassification1} we can choose isomorphisms $\gr_{E_i}(Z_i/Z_{i-1})\cong Z_{C_i}$ of graded $1$-truncated Dieudonn\'e modules and hence obtain bases $(e^i_j)_{j\in C_i}$ of the $\gr_{E_i}(Z_i/Z_{i-1})$. In the following by a pair $(i,j)$ we always mean such a pair satisfying $1\leq i\leq N$ and $j\in C_i$. For each pair $(i,j)$ let $f^i_j\in Z_i$ be a lift of $e^i_j$. 

Let $M$ be the free $\CO$-module with basis $(g_i^j)_{(i,j)}$. We make $M$ into a Dieudonn\'e module by defining the image of $g^i_j$ under $F$ and $V$ by a nested double induction, with the outer induction being increasing on $i$ and the inner induction being decreasing on $j$. For pairs $(i,j)$ and $(i',j')$ we let $(i,j)\prec (i',j')$ if and only if either the conditions $i=i'$ and $j>j'$ or the condition $i<i'$ is satisfied. 

First we define $F$: Consider a pair $(i,j)$. If $j+n_i \in C_i$ then $$F(f^i_j)=f^i_{j+n_i}+\sum_{(i',j')\prec (i,j+n_i)}a^{i'}_{j'}f^{i'}_{j'}$$ for certain $a^{i'}_{j'}\in k$. Then we let $$F(g^i_j)\defeq g^i_{j+n_i}+\sum_{(i',j')\prec (i,j+n_i)}[a^{i'}_{j'}]g^{i'}_{j'}.$$ Otherwise we have $j-m_i\in C_i$ and $$f^i_j=V(f^i_{j-m_i})+\sum_{(i',j')\prec (i,j)}b^{i'}_{j'}f^{i'}_{j'}$$ for certain $b^{i'}_{j'}\in k$. In this case we define $$F(f^i_j)\defeq \epsilon g^i_{j-m_i}+\sum_{(i',j')\prec (i,j)}[(b^{i'}_{j'})^p] F(g^{i'}_{j'}),$$ where the terms $F(g^{i'}_{j'})$ appearing are already defined by induction.

We define $V$ dually: Consider a pair $(i,j)$. If $j+m_i \in C_i$ then $$V(f^i_j)=f^i_{j+m_i}+\sum_{(i',j')\prec (i,j+m_i)}c^{i'}_{j'}f^{i'}_{j'}$$ for certain $c^{i'}_{j'}\in k$. Then we let $$V(g^i_j)\defeq g^i_{j+m_i}+\sum_{(i',j')\prec (i,j+m_i)}[c^{i'}_{j'}]g^{i'}_{j'}.$$ Otherwise we have $j-n_i\in C_i$ and $$f^i_j=F(f^i_{j-n_i})+\sum_{(i',j')\prec (i,j)}d^{i'}_{j'}f^{i'}_{j'}$$ for certain $d^{i'}_{j'}\in k$. In this case we define $$V(f^i_j)\defeq \epsilon g^i_{j-n_i}+\sum_{(i',j')\prec (i,j)}[(d^{i'}_{j'})^{-p}] V(g^{i'}_{j'}),$$ where the terms $V(g^{i'}_{j'})$ appearing are already defined by induction.

We extend $F$ and $V$ to a $\sigma$- respectively a $\sigma^{-1}$-linear endomorphism of $M$. 
\end{construction}
\begin{lemma} \label{SlopeLemma}
  Let $M$ be a Dieudonn\'e module and $n$ and $m$ non-negative integers such that the Newton polygon of $M$ has endpoint $(c n,c (n+m))$ for some integer $c\geq 0$. Assume that there exists a function $v\colon M\setminus \{0\} \to \BZ^{\geq 0}$ with the following properties:
  \begin{enumerate}[(i)]
  \item $v(F(x))=v(x)+n$ for all $x\in M$.
  \item $v(V(x))=v(x)+m$ for all $x\in M$.
  \end{enumerate}
Then $M$ is isoclinic of slope $n/(n+m)$.
\end{lemma}
\begin{proof}
  Let $\nu$ be a slope of $M$. There exists a non-zero Dieudonn\'e submodule $M'$ of $M$ such that for all integers $a\geq 0$ we have $F^{a(n_{\nu}+m_{\nu})}M'=p^{an_{\nu}}M'$. Let $x$ be a non-zero element of $M'$. For some integer $a\geq 0$, write $F^{a(n_\nu+m_\nu)}(x)=p^{an_\nu}(x')$ for some $x'\in M'$. Then we get:
\begin{equation*}
  v(x)+a(n_\nu+m_\nu)n=v(F^{a(n_\nu+m_\nu)}(x))=v(p^{an_{\nu}}(x'))=v(x')+an_{\nu}(n+m)\geq an_{\nu}(n+m)
\end{equation*}
By letting $a$ go to infinity this inequality implies $\nu=n_{\nu}/(n_{\nu}+m_{\nu})\leq n/(n+m)$. From this the claim follows by comparing the Newton polygon of $M$ to the constant Newton polygon of slope $n/(n+m)$ with the same endpoint.

\end{proof}
\begin{proposition} \label{LiftConsProps1}
  Let $Z$ and $E$ be as in Construction \ref{LiftCons1}. For each $1\leq i\leq N$ let $M_i \subset M$ be the $\CO$-submodule spanned by $\{g^{i'}_j\mid i'\leq i, j\in C_{i'}\}$.
  \begin{enumerate}[(i)]
  \item The module $M$ from Construction \ref{LiftCons1} is a Dieudonn\'e module, i.e. $FV=VF=\epsilon$.
  \item The assigment $g^i_j+\epsilon M \mapsto f^i_j$ gives an isomorphism $M/\epsilon M\cong Z$ of $1$-truncated Dieudonn\'e modules.
  \item The $M_i$ are Dieudonn\'e submodules of $M$.
  \item For each $1\leq i\leq N$, the Dieudonn\'e module $M_i/M_{i-1}$ is isoclinic of slope $n_i/(n_i+m_i)$.
  \item The Dieudonn\'e module $M$ has Newton polygon $\CP$.
  \end{enumerate}
\end{proposition}
\begin{proof}
  $(i)$, $(ii)$ and $(iii)$ follow from the definition of $M$ by the same double induction as in Construction \ref{LiftCons1}.

$(iv)$: We continue to use the notation from Construction \ref{LiftCons1}. For $j\in C_i$ we denote $f^i_j+M_{i-1}$ by $\bar f^i_j$. These elements form a $\CO$-basis of $M_i/M_{i-1}$. We define a function $$v\colon M_i/M_{i-1}\setminus \{0\} \to \BZ^{\geq 0}$$ by 
\begin{equation*}
  v(\sum_{j\in C_i}a_j \bar f^i_j)\defeq \min_{j\in C_i}((n_i+m_i)v(a_j)+j).
\end{equation*}
It follows from the definition of $M$ that $v$ satisfies the conditions of Lemma \ref{SlopeLemma} for $n=n_i$ and $m=m_i$. Thus $(iv)$ follows from Lemma \ref{SlopeLemma}.

$(v)$ follows from $(iv)$.

\end{proof}

\section{Existence of compatible flags} \label{MainThmSect}
Let $\CP=(\nu_1,\hdots,\nu_N)$ be a Newton polygon. For $1\leq i\leq N$ we denote $(n_{\nu_i},m_{\nu_i})$ by $(n_i,m_i)$ and let $h_i\defeq n_i+m_i$ and $d_i\defeq m_i$. For such $i$ we let $G_i, T_i, W_i, \CI_i,\mu_i$, etc., be the data from Subsection \ref{sec:setup} associated to $(h,d)=(h_i,d_i)$. Let $h=\sum_i h_i$ and $\prod_{1\leq i\leq N}G_i\cong H\subset G=\GL_h$ be the Levi subgroup containing $T$ corresponding to the decomposition $h=h_1+\hdots+h_N$. We denote by $\tilde W_{H}\defeq H(W(k))\cap \tilde W$ (resp. $W_H$) the extended Weyl group (resp. the Weyl group) of $H$. Let $d\defeq \sum_{1\leq i\leq N}n_i$.

\begin{definition}
  Let $\lambda\in X_*(T)$. There is a unique permutation $\eta\in W_{H}$ with the following properties:
  \begin{enumerate}[(i)]
  \item For each $1\leq i\leq N$ we have $\lambda_{\eta(h_1+\hdots+h_{i-1}+1)}\leq \lambda_{\eta(h_1+\hdots+h_{i-1}+2)} \leq \hdots \leq \lambda_{\eta(h_1+\hdots+h_i)}$.
  \item For each $1\leq j,j' \leq h$ such that $\lambda_{j}=\lambda_{j'}$ we have $j<j'$ if and only if $\eta(j)<\eta(j')$.
  \end{enumerate}
We denote this permutation $\eta$ by $\eta_\lambda$.
\end{definition}

\begin{definition}
  Let $x_\CP\in \tilde W_{H}$ be the matrix whose $i$-th block is given by $x_{n_i,m_i}$ for each $1\leq i\leq N$.
\end{definition}
\begin{theorem} \label{CompFlagExists}
  Let $M$ be a Dieudonn\'e module with Hodge polygon given by $\mu$. The following are equivalent:
  \begin{enumerate}[(i)]
  \item On the truncation $Z=M/\epsilon M$ there exists a compatible filtration with Newton polygon $\CP$.
  \item There exists $\lambda\in X_*(T)$ satisfying $\epsilon^{-\lambda}x_\CP\epsilon^\lambda\in W\epsilon^\mu W$ such that the matrix of $F\colon M\to M$ with respect to some $\CO$-basis of $M$ lies in $\CI\eta_\lambda^{-1}\epsilon^{-\lambda}x_\CP\epsilon^\lambda\eta_\lambda \CI$.
  \end{enumerate}
\end{theorem}
\begin{proof}
Using $\sigma$-conjugation by elements of $G(\CO)$, which amounts to base change on $M$, one sees that $(ii)$ is equivalent to saying that there exists such a $\lambda$ such that the matrix of $F$ with respect to some $\CO$-basis of $M$ lies in $\leftexp{n_\lambda}{\CI}\epsilon^{-\lambda}x_\CP\epsilon^\lambda$. 

$(i)\implies (ii)$: Let $E=((Z_i)_{0\leq i\leq N},(E_i)_{1\leq i\leq N})$ be a compatible filtration of Newton polygon $\CP$ on $Z$. Fix $1\leq i\leq N$. By Proposition \ref{ShortClassification2} there exists $\lambda^i\in X_*(T_i)$ satisfying $\epsilon^{-\lambda^i}x_{n_i,m_i}\epsilon^{\lambda^i}\in W_i\epsilon^{\mu_i} W_i$ such that $\gr_{E_i}(Z_i/Z_{i-1})\cong Z_{\lambda^i}$. Let $M^i$ and $(f^i_j)_{1\leq i\leq h_i}$ be the Dieudonn\'e module together with its $\CO$-basis from Construction \ref{ZlambdaCons} applied to $\lambda=\lambda^i$ such that $Z_{\lambda^i}=M^i/\epsilon M^i$ and the matrix of $F\colon M^i\to M^i$ with respect to $(f^i_j)_{1\leq j\leq h_i}$ is $\epsilon^{-\lambda^i}x_{n_i,m_i}\epsilon^{\lambda^i}$. Fix an isomorphism $M^i/\epsilon M^i\cong \gr_{E_i}(Z_i/Z_{i-1})$ and let $(\bar f^i_j)_{1\leq j\leq h_i}$ be the image of $(f^i_j)_{1\leq j\leq h_i}$ in $\gr_{E_i}(Z_i/Z_{i-1})$. Let $M_i$ be the preimage of $Z_i$ in $M$ and for $1\leq j\leq h_i$ let $\tilde f^i_j$ be lift of $\bar f^i_j$ to $Z_i$ and $g^i_j$ a lift of $\tilde f^i_j$ to $M_i$. 

By comparing the definition of $Z_{\lambda^i}$ and $\eta_{\lambda^i}$ one sees that the subspaces appearing in the filtration $E_i$ on $Z_i/Z_{i-1}$ are those of the form $\sum_{1\leq j'\leq j}k\tilde f^i_{\eta_{\lambda^i}(j')}+Z_{i-1}$ for $1\leq j\leq h_i$. This together with the fact that the matrix of $F\colon M^i\to M^i$ with respect to $(f^i_j)_{1\leq j\leq h_i}$ is $\epsilon^{-\lambda_i}x_{n_i,m_i}\epsilon^{\lambda_i}$ implies that the matrix of $F\colon M_i/M_{i-1}\to M_i/M_{i-1}$ with respect to the basis $(g^i_j)_{1\leq j\leq h_i}$ lies in $\leftexp{\eta_{\lambda_i}}{\CI} \epsilon^{-\lambda_i}x_{n_i,m_i}\epsilon^{\lambda_i}$.

Now let $\lambda\in X_*(T)$ be the cocharacter whose factor in the $i$-th block of $H$ is given by $\lambda^i$ for each $1\leq i\leq N$. From the definition of $x_\CP$ and the corresponding property of the $\lambda^i$ it follows that $\epsilon^{-\lambda}x_\CP\epsilon^\lambda\in W\epsilon^\mu W$. Furthermore, from the definition of $\eta_\lambda$ and the above it follows that the matrix of $F\colon M\to M$ with respect to the $\CO$-basis $(f^i_j)_{i,j}$ lies in $\leftexp{\eta_\lambda}{\CI}\epsilon^{-\lambda}x_\CP\epsilon^\lambda$. This proves $(ii)$.

$(ii)\implies (i)$: We reverse the above arguments: By assumption there exists a $\CO$-basis of $M$ with respect to which the matrix of $F$ lies in $\leftexp{\eta_\lambda}{\CI}\epsilon^{-\lambda}x_\CP\epsilon^\lambda$. Write such a basis as $(f^1_1,f^1_2,\hdots, f^1_{h_i}, f^2_1,\hdots,f^N_{h_N})$. For $1\leq i\leq N$ let $M_i\defeq \sum_{i'\leq i, j}\CO f^{i'}_j$ and $Z_i$ the image of $M_i$ in $Z$. The form of the matrix of $F$ with respect to the basis $(f^i_j)_{(i,j)}$ implies that $F(M_i)\subset M_i$ for each $i$. Fix $1\leq i\leq N$. Let $\lambda^i$ (resp. $\eta_{\lambda^i}$) be the part of $\lambda$ (resp. $\eta_{\lambda}$) in $G_i$. Then the matrix of $F$ on $M_i/M_{i-1}$ with respect to $(g^i_j)_{1\leq j\leq h_i}$ lies in $\leftexp{\eta_{\lambda^i}}{\CI_i}\epsilon^{-\lambda^i}x_{n_i,m_i}\epsilon^{\lambda^i}$ which proves that $M_i/M_{i-1}$ is a Dieudonn\'e module with Hodge polygon given by $\mu_i$ and hence that $Z_i/Z_{i-1}$ is a $1$-truncated Dieudonn\'e module of rank $h_i$.

 For $1\leq j\leq h_i$ let $\tilde f^i_j$ be the image of $g^i_j$ in $Z_i$. As above we consider the graded $1$-truncated Dieudonn\'e module $Z_{\lambda^i}$ with its canonical basis $(\bar f^i_j)_{1\leq i\leq j}$. Let $(G^j(Z_{\lambda^i}))_{j\in \BZ}$ be the canonical filtration of type $(n_i,m_i)$ associated to the grading on $Z_{\lambda^i}$. For $j\in \BZ$ define $G^j(Z_i/Z_{i-1})\defeq \sum_{\{j'\colon \bar f^i_{j'}\in G^j(Z_{\lambda^i})\}}k\tilde f^i_{j'}$. Similar to the above one checks by comparison with $Z_{\lambda^i}$ that this defines a compatible filtration $E_i$ of type $(n_i,m_i)$ on $Z_i/Z_{i-1}$. Alltogether we have constructed a compatible filtration with Newton polygon $\CP$ on $Z$.

\end{proof}

Now we can prove our main result:
\begin{theorem} \label{MainTheorem}
  Let $w\in \leftexp{I}{W}$. The following are equivalent:
  \begin{enumerate}[(i)]
  \item The $1$-truncated Dieudonn\'e module $Z_w$ admits a lift with Newton polygon $\CP$.
  \item On $Z_w$ there exists a compatible filtration with Newton polygon $\CP$.
  \item There exists $\lambda\in X_*(T)$ satisfying $\epsilon^{-\lambda}x_\CP\epsilon^\lambda\in W\epsilon^\mu W$ such that $ww_0w_{0,I}\epsilon^\mu$ is $G(\CO)$-$\sigma$-conjugate to an element of $\CI \eta_\lambda^{-1}\epsilon^{-\lambda}x_\CP\epsilon^\lambda\eta_\lambda \CI$.
  \item There exist $\lambda\in X_*(T)$ satisfying $\epsilon^{-\lambda}x_\CP\epsilon^\lambda\in W\epsilon^\mu W$ as well as $y\in W$ such that $ww_0w_{0,I}\epsilon^\mu\in \CI y\CI \eta_\lambda^{-1}\epsilon^{-\lambda}x_\CP\epsilon^\lambda\eta_\lambda \CI y^{-1} \CI$.
  \end{enumerate}
\end{theorem}
\begin{proof}
  The implication $(i)\implies (ii)$ follows from Construction \ref{FilRedCons2}. The implication $(ii)\implies (i)$ follows from Proposition \ref{FilRedCons2}. The equivalence of $(ii)$ and $(iii)$ is a reformulation of Theorem \ref{CompFlagExists} applied to the Dieudonn\'e module $M_{ww_0w_{0,I}\epsilon^\mu}$. 

The implication $(iii)\implies (iv)$ follows from the decomposition $G(\CO)=\coprod_{y\in W}\CI y \CI$. If $(iv)$ holds, there exists an element of $\CI ww_0w_{0,I}\epsilon^\mu \CI$ which is $G(\CO)$-$\sigma$-conjugate to an element of $\CI \eta_\lambda^{-1}\epsilon^{-\lambda}x_\CP\epsilon^\lambda\eta_\lambda \CI$. By \cite[Theorem 1.1]{ViehmannTruncations}, each element of $\CI ww_0w_{0,I}\epsilon^\mu \CI$ is $G(\CO)$-$\sigma$-conjugate to an element of $G(\CO)_1 ww_0w_{0,I}\epsilon^\mu G(\CO)_1$. Using the fact that $G(\CO)_1$ is normal in $G(\CO)$ this implies $(iii)$.
\end{proof}

Let $\CZ$ be the center of $H$. Then $X_*(\CZ)$ acts on the set
\begin{equation*}
  X_*(T)^\CP\defeq \{\lambda\in X_*(T)\mid \epsilon^{-\lambda}x_{\CP}\epsilon^\lambda\in W\epsilon^\mu W\}
\end{equation*}
by addition. 
\begin{lemma} \label{LambdaFiniteness}
  This action on $X_*(T)^\CP$ has finitely many orbits.
\end{lemma}
\begin{proof}
  By looking at each block of $H$ separately, we assume that $\CP=(n/(n+m))$ for coprime non-negative integers $n$ and $m$. Via Propositions \ref{ShortClassification1} and \ref{ShortClassification2}, the set $X_*(T)^\CP$ can be identified with the set of beginnings $C$ of semimodules of type $(n,m)$. Under this identification, an element $i\in X_*(\CZ)\cong \BZ$ sends $C\subset \BZ$ to $C+i$. In this form the claim is \cite[6.3]{JOPurity}.
\end{proof}
\section{Non-emptiness of certain affine Deligne-Lusztig varieties}
\label{DLSection}
Fix $0\leq d\leq h$. For $x\in \tilde W$ and $b\in G(L)$, we consider the associated affine Deligne-Lusztig variety (c.f. Rapoport \cite{RapoportGuide}), which is the following set:
\begin{equation*}
  X_{x}(b)\defeq \{g\CI \in G(L)/\CI \mid g^{-1}b\sigma(g) \in \CI x \CI\}
\end{equation*}

From Theorem \ref{MainTheorem} we get the follwing criterion for the non-emptiness of certain of the $X_x(b)$. Here we use again the objects defined in Section \ref{MainThmSect} with respect to the given Newton polygon $\CP$. In case the Newton polygon $\CP$ has a single slope, a different such criterion was previously given by G\"ortz, He and Nie in \cite{GHN}.
\begin{theorem}
  Let $x\in  W\epsilon^\mu W$ and $b\in G(\CO)\epsilon^\mu G(\CO)$. Let $\CP$ the Newton polygon of the Dieudonn\'e module $M_b$. The following are equivalent:
  \begin{enumerate}[(i)]
  \item The set $X_x(b)$ is non-empty.
  \item There exist $\lambda\in X_*(T)$ satisfying $\epsilon^{-\lambda}x_\CP\epsilon^\lambda\in W\epsilon^\mu W$ and $y \in W$ such that $$x\in \CI y\CI \eta_\lambda^{-1} \epsilon^{-\lambda}x_\CP\epsilon^\lambda\eta_\lambda \CI y^{-1} \CI.$$ 
  \end{enumerate}
\end{theorem}
\begin{proof}
  $(i)\implies (ii)$: Let $g\CI \in X_x(b)$ and $h\defeq gb\sigma(g^{-1})\in \CI x\CI$. Since $x\in W\epsilon^\mu W$ we obtain a Dieudonn\'e module $M_h$ with Hodge polygon given by $\mu$ and Newton polygon $\CP$. Hence by Theroem \ref{MainTheorem} there exists a compatible filtration with Newton polygon $\CP$ on $M_h/\epsilon M_h$. Hence by Theorem \ref{CompFlagExists} applied to $M=M_h$ there exist $\lambda$ as in $(ii)$ and $r\in G(\CO)$ such that $r h\sigma(r)^{-1}\in \CI \eta_\lambda \epsilon^{-\lambda}x_\CP\epsilon^\lambda \eta_\lambda \CI$. Using $G(\CO)=\coprod_{y\in W}\CI w\CI$ this proves $(ii)$.

$(ii)\implies (i)$: By $(ii)$ there exists an element $h\in \CI x\CI$ which is $G(\CO)$-$\sigma$-conjugate to an element of $\CI \eta_\lambda^{-1} \epsilon^{-\lambda}x_\CP\epsilon^\lambda\eta_\lambda \CI$. Hence by Theorem \ref{MainTheorem} the $1$-truncated Dieudonn\'e module $Z\defeq M_h/\epsilon M_h$ has a lift $M$ with Newton polygon $\CP$. Since $M$ and $M_h$ have the same truncation, as discussed in Subsection \ref{TruncClassSect} the matrix $h'$ of $F\colon M\to M$ with respect to a suitable basis lies in $G(\CO)_1h G(\CO)_1$. Since $G(\CO)_1\subset \CI$ we have $h'\in \CI x\CI$. Since $M_{h'}\cong M$ has Newton polygon $\CP$ there exists $g\in G(L)$ such that $g^{-1}b\sigma(g)=h'\in \CI x\CI$. Thus $g\CI \in X_x(b)$.
\end{proof}

%%%%%%%%%%%%%%%%%%%%%%%%%%%%%%%%%%%%%%%%%%
\bibliography{references}
\bibliographystyle{alphanum}
%%%%%%%%%%%%%%%%%%%%%%%%%%%%%%%%%%%%%%%%%%%%%%%%%%%%%%%%%%%%%%%%%%%%%%%%%%%%%%%%%%%%%%%%%%%

\end{document}